\documentclass[11pt]{article}
\usepackage{amsmath,amssymb,amsfonts,euscript,pifont, graphicx, color}
\usepackage{url}
\def\G{{\Gamma}}

\def\1{{\bf 1}}
\def\0{{\bf 0}}

\def\L{{\mathcal{L}}}
\def\D{{\mathcal{D}}}

\newtheorem{prethm}{{\bf Theorem}}

\newenvironment{thm}{\begin{prethm}\sl{\hspace{-0.5
               em}{\bf}}}{\end{prethm}}
\newtheorem{preprop}{{\bf Proposition}}

\newenvironment{prop}{\begin{preprop}\sl{\hspace{-0.5
               em}{\bf}}}{\end{preprop}}
\newtheorem{preex}{{\bf Example}}

\newtheorem{precor}{{\bf Corollary}}

\newenvironment{cor}{\begin{precor}\sl{\hspace{-0.5
               em}{\bf}}}{\end{precor}}

\newtheorem{preremark}{{\bf Remark}}

\newtheorem{prelem}{{\bf Lemma}}

\newenvironment{lem}{\begin{prelem}\sl{\hspace{-0.5
               em}{\bf}}}{\end{prelem}}

\newtheorem{preproof}{{\bf Proof}}

\newenvironment{proof}[1]{\begin{preproof}{\rm
               #1}\hfill{$\Box$}}{\end{preproof}}

\title{Graphs whose normalized Laplacian has three eigenvalues\footnote{This version is published in Linear Algebra and its Applications 435 (2011), 2560-2569.}}

\author{E.R. van Dam$^{\textrm{a}}$   \quad  G.R. Omidi$^{\textrm{b},\textrm{c},1}$  \\[2pt]
{\small  $^{\textrm{a}}$Tilburg University, Dept. Econometrics and Operations Research,}\\
{\small  P.O. Box 90153, 5000\,LE, Tilburg, The Netherlands}\\
{\small  $^{\textrm{b}}$Dept. Mathematical Sciences, Isfahan University of Technology},\\
{\small Isfahan, 84156-83111, Iran}\\
{\small $^{\textrm{c}}$School of Mathematics, Institute for Research in Fundamental Sciences (IPM),}\\
{\small P.O. Box 19395-5746, Tehran, Iran }\\[2pt]
{$\mathsf{edwin.vandam@uvt.nl}$ \quad  $\mathsf{romidi@cc.iut.ac.ir}$  }}

\date{}

\begin{document}

\maketitle \footnotetext[1] {This research was in part
supported by a grant from IPM (No. 89050037).}

\begin{abstract}
\noindent We give a combinatorial characterization of graphs
whose normalized Laplacian has three distinct eigenvalues.
Strongly regular graphs and complete bipartite graphs are
examples of such graphs, but we also construct more exotic
families of examples from conference graphs, projective planes,
and certain quasi-symmetric designs.
\end{abstract}
\\[.1cm]

\section{Introduction}
In their pioneering monograph on spectra of graphs,
Cvetkovi\'c, Doob, and Sachs \cite[\S 1.2, 1.6]{CDS} mention
the spectrum of the transition matrix as one of the possible
spectra to investigate graphs, and they give some properties of
the coefficients of the corresponding characteristic
polynomial. The spectrum of the transition matrix and the
spectrum of the normalized Laplacian matrix are in (an easy)
one-one correspondence, so that studying the latter is
essentially the same as studying the first. The normalized
Laplacian is mentioned briefly in the recent monograph by
Cvetkovi\'c, Rowlinson, and Simi\'c \cite[\S 7.7]{CRS};
however, the standard reference for it is the monograph by
Chung \cite{Fc}, which deals almost entirely with this matrix.

Graphs with few distinct eigenvalues have been studied for
several matrices, such as the adjacency matrix
\cite{BM,dCvDS,CO,vD-4e,vD-3e,vDS,MK}, the Laplacian matrix
\cite{vDH-mu,WFT}, the signless Laplacian matrix \cite{AOT},
the Seidel matrix \cite{S2}, and the universal adjacency matrix
\cite{HO1}. One of the reasons for studying such graphs is that they
have a lot of structure, and can be thought of as
generalizations of strongly regular graphs (see also the
manuscript by Brouwer and Haemers \cite{BH}).

Typically, graphs with few distinct eigenvalues seem to be the
hardest graphs to distinguish by the spectrum. Put a bit
differently, it seems that most graphs with few eigenvalues are
not determined by the spectrum. Thus, the question of which
graphs are determined by the spectrum (as studied in
\cite{vDH-ds, DS2}) is another motivation for studying graphs
with few distinct eigenvalues. For the normalized Laplacian
matrix, there are some recent constructions of graphs with the
same spectrum by Butler and Grout \cite{B,BG}. Some other
recent work on the normalized Laplacian (energy) is done by
Cavers, Fallat, and Kirkland \cite{CFK}.

In this paper, we investigate graphs whose normalized
Laplacian has three eigenvalues. The only graphs whose
normalized Laplacian has one eigenvalue are empty graphs, and
the (connected) ones with two eigenvalues are complete. We
shall give a characterization of graphs whose normalized
Laplacian has three eigenvalues. Strongly regular graphs and
complete bipartite graphs are examples of such graphs, but we
also construct more exotic families of examples from conference
graphs, projective planes, and certain quasi-symmetric designs.

\section{Basics}

Throughout, $\G$ will denote a simple undirected graph with $n$
vertices. The {\sl adjacency matrix} of $\G$ is the $n\times n$
$01$-matrix $A=[a_{uv}]$ with rows and columns indexed by the
vertices, where $a_{uv}=1$ if $u$ is adjacent to $v$, and $0$
otherwise. Let $D=[d_{uv}]$ be the $n\times n$ diagonal matrix
where $d_{uu}$ equals the valency $d_{u}$ of vertex $u$. The
matrix $L = D-A$ is better known as the {\sl Laplacian matrix}
of $\G$. The {\sl normalized Laplacian} matrix of $\G$ is the
$n\times n$ matrix $\mathcal{L}=[\ell_{uv}]$ with
$$\ell_{uv}=\left \{ \begin{array}{cl} 1 &{\rm if} ~u=v, ~
d_u\neq 0, \\
 -1/\sqrt{d_{u}d_{v}} & {\rm if} ~u {\rm ~is ~adjacent~ to~} v, ~\\
 0 & {\rm otherwise.}  \end{array}\right. $$
If $\G$ has no isolated vertices then
$\mathcal{L}=D^{-\frac{1}{2}}LD^{-\frac{1}{2}}=I-D^{-\frac{1}{2}}AD^{-\frac{1}{2}}$.
Mohar \cite{M} calls this matrix the transition Laplacian, but
others (for example Tan \cite{T}) use this term for the matrix
$D^{-1}L$. Both matrices, and also the transition matrix
$D^{-1}A$, have the same number of distinct eigenvalues, so for
our purpose this makes no difference. Let $\lambda_1\geq
\lambda_2\geq \cdots \geq \lambda_n$ be the eigenvalues of
$\mathcal{L}$ or, as we shall write from now on, the
$\mathcal{L}$-eigenvalues of $\G$. The following basic results
are from \cite[Lemmas 1.7-8]{Fc} (see also \cite{CDHLPS}).

\begin{lem}\label{basic} Let $n \geq 2$. A graph $\G$ on $n$ vertices has the following properties.\\
{\rm (i)} $\lambda_n=0$,\\{\rm (ii)}  $\sum_i
\lambda_i\leq n$ with equality holding if and only if $\G$ has no
isolated vertices,\\{\rm (iii)}  $\lambda_{n-1}\leq
n/(n-1)$ with equality holding if and only if $\G$ is a complete
graph on $n$ vertices,\\{\rm (iv)} $\lambda_{n-1}\leq 1$ if $\G$ is non-complete,\\
{\rm (v)} $\lambda_{1}\geq n/(n-1)$ if $\G$ has no isolated
vertices,\\{\rm (vi)} $\lambda_{n-1}> 0$ if
$\G$ is connected. If $\lambda_{n-i+1}=0$
and $\lambda_{n-i}\neq 0$, then $\G$ has exactly $i$ connected
components,\\{\rm (vii)} The spectrum of $\G$
is the union of the spectra of its connected components,\\{\rm (viii)} $\lambda_i\leq 2$ for all $i$,
with $\lambda_1=2$ if and only if some connected component of $\G$ is a
non-trivial bipartite graph,\\{\rm (ix)}
$\G$ is bipartite if and only if $2-\lambda_i$ is an eigenvalue of $\G$ for each $i$.
\end{lem}
Because of (vii), the study of the $\mathcal{L}$-eigenvalues
can be restricted to connected graphs without loss of
generality. So from now on, $\G$ will be a connected graph, and
the trivial $\L$-eigenvalue $0$ occurs with multiplicity one.

\section{Three distinct eigenvalues}

In this section, we give a characterization of graphs whose
normalized Laplacian has three (distinct) eigenvalues. This characterization forms the basis
for the rest of the paper.
Using Lemma \ref{basic}, it follows that the only graphs with one
$\mathcal{L}$-eigenvalue are the empty graphs.
Using (ii) and {\rm (iii)} of Lemma \ref{basic}, we find that a connected graph
has two $\mathcal{L}$-eigenvalues if and only if it is complete.

In order to describe graphs with three normalized Laplacian
eigenvalues, we let $\hat{d}_u=\sum_{v\sim u}\frac1{d_v}$ be the
{\em normalized valency} of $u$, and let $\sum_{w\sim
u,v}\frac1{d_w}$  be the {\em normalized number of common neighbors} of
two distinct vertices $u$ and $v$. We denote this normalized number of common neighbors by
$\hat{\lambda}_{uv}$ if $u$ and $v$ are adjacent, and by $\hat{\mu}_{uv}$ if they are not.

\begin{thm}\label{main} Let $\Gamma$ be a connected graph with $e$ edges.
Then $\Gamma$ has three $\mathcal{L}$-eigenvalues
$0, \theta_1, \theta_2$ if and only if the
following three properties hold.\\
{\rm (i)} $\hat{d}_u=td_u^2-(\theta_1-1)(\theta_2-1)d_u$ for all
vertices $u$,\\
{\rm (ii)} $\hat{\lambda}_{uv}=td_ud_v+2-\theta_1-\theta_2$ for
adjacent vertices $u$ and $v$,\\
{\rm (iii)} $\hat{\mu}_{uv}=td_ud_v$ for non-adjacent vertices $u$ and
$v$,\\
where $t=\frac{\theta_1\theta_2}{2e}$.
\end{thm}

\begin{proof} {Since $\mathcal{L}$ is
symmetric, it follows that $\G$ has eigenvalues $0, \theta_1$,
and $\theta_2$ if and only if
$(\mathcal{L}-\theta_1I)(\mathcal{L}-\theta_2I)$ is a symmetric
rank one matrix. If so, then its non-zero eigenvalue is
$\theta_1\theta_2$ and has eigenvector $D^{\frac12}{\bf j}$ (an
eigenvector of $\mathcal{L}$ corresponding to eigenvalue $0$),
where {\bf j} is the all-ones vector. By working this out, we
get the equation
$$(\mathcal{L}-\theta_1I)(\mathcal{L}-\theta_2I)=\frac{\theta_1\theta_2}{2e}(D^{\frac12}{\rm
{\bf j}})(D^{\frac12}{\rm {\bf j}})^{\top}.$$ From this
equation, the stated characterization follows. }
\end{proof}
From Theorem \ref{main} we immediately find the below two corollaries.
Recall that $\G$ is {\sl strongly regular} with
parameters $(n,k,\lambda,\mu)$, whenever $\G$ is $k$-regular with
$0<k<n-1$, and the number of common neighbors of any two distinct
vertices equals $\lambda$ if the vertices are adjacent and $\mu$
otherwise (see \cite{BH}).

\begin{cor}\label{srg}
A connected regular graph has three
$\mathcal{L}$-eigenvalues if and only if it is strongly regular.
\end{cor}

\begin{cor}\label{diameter}
A connected graph with three $\mathcal{L}$-eigenvalues has diameter
two.
\end{cor}
Both results are not surprising, knowing that the same results
hold for other matrices such as the adjacency matrix,
Laplacian matrix, and signless Laplacian matrix.

\section{Bipartite graphs}

A complete bipartite graph is an example of a graph with three $\mathcal{L}$-eigenvalues;
it was already observed by Chung \cite[Ex. 1.2]{Fc} that it has eigenvalues
$0,1$ (with multiplicity $n-2$), and $2$.
In this section, we give some characterizations
of bipartite graphs with three $\mathcal{L}$-eigenvalues.

\begin{prop} Let $\G$ be a connected triangle-free graph with
three $\mathcal{L}$-eigenvalues. Then $\G$ is a triangle-free
strongly regular graph or a complete bipartite graph. \end{prop}

\begin{proof} { If $\G$ is regular, then it is clearly strongly-regular. So
assume that $\G$ is non-regular. Because $\G$ is triangle-free,
and using Theorem \ref{main}, it follows that for every pair of
adjacent vertices $u,v$, it holds that
$0=\hat{\lambda}_{uv}=td_ud_v+2-\theta_1-\theta_2$. Because $G$
is connected and non-regular, there is a pair of adjacent
vertices $u,v$ with distinct valencies $d_u$ and $d_v$. The
above equation now implies that only these two valencies occur,
and that there are no odd cycles in $\G$. Hence $\G$ is
bipartite. By Corollary \ref{diameter}, $\G$ must be complete
bipartite. }\end{proof}
We call a graph $\mathcal{L}$-integral if all its $\mathcal{L}$-eigenvalues are integral, i.e., 0, 1, or 2.
The complete bipartite graphs are such graphs; in fact, no other connected graphs are $\mathcal{L}$-integral.

\begin{prop}\label{bipartite} Let $\G$ be connected. Then the following are equivalent.\\
{\rm (i)} $\G$ is bipartite with three $\mathcal{L}$-eigenvalues,\\
{\rm (ii)} $\G$ is $\mathcal{L}$-integral,\\
{\rm (iii)} $\G$ is complete bipartite.
\end{prop}
\begin{proof} {First we show that {\rm (i)} implies {\rm (ii)}. Let $\G$ be bipartite with three
$\mathcal{L}$-eigenvalues. By {\rm (i)} and {\rm (ix)} of Lemma
\ref{basic}, it clearly follows that $\G$ has $\L$-eigenvalues 0, 1, and 2, and hence it is integral.

Next we show that {\rm (ii)} implies {\rm (iii)}. Let $\G$ be
integral. By Lemma \ref{basic}, the $\mathcal{L}$-spectrum of
$\G$ is $\{[0]^1, [1]^{n-2}, [2]^{1}\}$ and hence $\G$ is
bipartite. Because its diameter equals two, $\G$ is complete
bipartite. It is clear that we can conclude {\rm (i)} from {\rm
(iii)}.
 }\end{proof}
The property that complete bipartite graphs have two simple eigenvalues does not characterize them among
the graphs with three $\L$-eigenvalues, as we shall see later on.

\section{Biregular graphs}

In this section, we shall consider biregular graphs with three distinct
$\mathcal{L}$-eigenvalues. We call a graph with two distinct valencies $k_1$ and $k_2$ {\sl
$(k_1,k_2)$-regular}, or simply {\sl biregular}. The complete bipartite graphs of the previous section
are examples of biregular (or regular) graphs.
Characterization of the biregular graphs with three distinct
$\mathcal{L}$-eigenvalues seems to be difficult though, so we shall have a look at some
special cases (also in the next section).

\subsection{The valency partition}

A partition $\sigma=\{V_1,...,V_m\}$ of
the vertex set of a graph $\G$ is called an {\sl equitable
partition} if for all $i,j=1,...,m$, the number of neighbors in $V_j$ of $u\in V_i$
depends only on $i,j$, and not on $u$; we denote this number by $k_{ij}$. We call
the partition of the vertex set according to valencies the {\sl valency partition}.
The following can be obtained from Theorem \ref{main}.

\begin{lem}\label{biregular} Let $\G$ be a biregular graph with three
$\mathcal{L}$-eigenvalues. Then the valency partition is
equitable.
\end{lem}

\begin{proof} {Suppose $\G$ is $(k_1,k_2)$-regular, and let $V_i$, $i=1,2$, be the set of vertices of
valency $k_i$. Fix a vertex $u\in V_i$, and let $k_{ij}$ be the number of neighbors in $V_j$ of $u\in V_i$. It
follows that these numbers do not depend on the particular $u$ because they are determined by the equations
$k_{i1}+k_{i2}=k_i$ and $\frac{k_{i1}}{k_1}+\frac{k_{i2}}{k_2}=\hat{d}_u=tk_i^2-(\theta_1-1)(\theta_2-1)k_i$.
}
\end{proof}

\subsection{Projective planes}

To find more examples of graphs with three
$\mathcal{L}$-eigenvalues we let $\G$ be such a
$(k_1,k_2)$-regular graph, with valency partition
$\{V_1,V_2\}$, and we suppose that the induced subgraph $\G_1$
on $V_1$ is empty. By Lemma \ref{biregular}, $\{V_1,V_2\}$ is
an equitable partition and by Theorem \ref{main}, the number of
common neighbors of any two vertices in $V_1$ is a
constant $tk_1^2k_2$ (which is $k_2$ times the normalized number of common neighbors).
Hence we may assume
that the bipartite graph between $V_1$ and $V_2$ is the
incidence graph of a $2$-design $\D$. Now it is convenient to switch to notation
that is common in design theory. So we let $v=|V_1|$, $b=|V_2|$, $k=k_{21}$, $r=k_1$, and $\lambda=tk_1^2k_2$,
so that $\D$ is a $2$-$(v,k,\lambda)$ design with $b$ blocks and replication number $r$.
In case $V_1$ and $V_2$ have the same size,
then this design is symmetric, and we obtain the following.

\begin{prop}\label{projective} Let $\G$ be a non-bipartite biregular graph such that
the valency partition has parts of equal size, and
the induced graph on one of the parts is empty. Then $\G$ has three
$\mathcal{L}$-eigenvalues if and only if it is obtained from the incidence graph of a
projective plane by making any two vertices corresponding to the
lines adjacent. If the projective plane has line size $k$ and $v=k^2-k+1$ points, then the
non-trivial $\mathcal{L}$-eigenvalues of $\G$ are $\frac{v}{k^2}$ and $1+\frac1{k}$.
\end{prop}

\begin{proof} {We continue with the above notation and arguments, and assume that $\G$ has three $\mathcal{L}$-eigenvalues. The design $\D$
is a symmetric $2$-$(v,k,\lambda)$ design, and $r=k$. Furthermore, let $a=k_{22}$. From Theorem \ref{main}, we obtain
the equations
\begin{equation}\label{e1}
\frac{k}{k+a}=tk^2-(\theta_1-1)(\theta_2-1)k,
\end{equation}
\begin{equation}\label{e2}
\frac{\lambda}{k+a}=tk^2,
\end{equation}
\begin{equation}\label{e3}
1+\frac{a}{k+a}=t(k+a)^2-(\theta_1-1)(\theta_2-1)(k+a).
\end{equation}
By combining these three equations, we find that $t=\frac{1}{(k+a)^2}$ (note that $a \neq 0$ because $\G$ is not bipartite) and $\lambda=\frac{k^2}{k+a}$. The latter implies that $\lambda \neq k$, otherwise $\G$ would be regular and bipartite. Thus, $\D$ is not a complete design. Because $\D$ is also not empty, we obtain two more equations from Theorem \ref{main}:
\begin{equation}\label{e4}
\frac{\mu_{12}}{k+a}=tk(k+a),
\end{equation}
\begin{equation}\label{e5}
\frac{\lambda_{12}}{k+a}=tk(k+a)+2-\theta_1-\theta_2,
\end{equation}
where $\lambda_{12}$ and $\mu_{12}$ are the numbers of common neighbors of a vertex in $V_1$ and a vertex in $V_2$, depending on whether they are adjacent or not, respectively. It follows that $\mu_{12}=k$, and we claim that this implies that the induced graph $\G_2$ on $V_2$ is complete. To show this claim, consider two blocks (vertices in $V_2$)
$B_1$ and $B_2$. Because the design is not complete, there is a point $P$ that is incident with $B_1$, but not with $B_2$. Because $\mu_{12}=k$, every neighbor of $P$ is also a neighbor of $B_2$; in particular this holds for $B_1$, which proves our claim. Thus, $a=k_{22}=v-1$, $k_2=k+v-1$, and $\lambda=\frac{k^2}{k+v-1}$. When the latter is combined with the property that $\lambda(v-1)=k(k-1)$ (because $\D$ is a symmetric design), we obtain that $v=k^2-k+1$ and $\lambda=1$, i.e., $\D$ is a projective plane, and $\G$ is as stated. Moreover, $\lambda_{12}=k-1$, and so (\ref{e4}) and (\ref{e5}) imply that $\theta_1+\theta_2-2=\frac{1}{k+v-1}=\frac1{k^2}$. Together with the equation $(\theta_1-1)(\theta_2-1)=\frac{k}{(k+v-1)^2}-\frac1{(k+v-1)}=\frac1{k^3}-\frac1{k^2}$, which follows from (\ref{e1}), this determines the non-trivial $\mathcal{L}$-eigenvalues. On the other hand, if $\G$ is as stated, then the equations from Theorem \ref{main} all hold, including a final one that was not used so far:
\begin{equation}\label{e6}
\frac{1}{k}+\frac{v-2}{k+v-1}=t(k+v-1)^2+2-\theta_1-\theta_2.
\end{equation}
}\end{proof}

\subsection{Quasi-symmetric designs}

\noindent If in the above discussion the design $\D_{\G}$ is not symmetric, then it seems hard to
characterize $\G$. In this case $\G_2$ cannot be empty (unless $\G$ is complete bipartite) or complete.
It seems natural to consider the case that $\G_2$ is strongly regular, and indeed, there are such examples as we shall see. In this case, it follows from Theorem \ref{main} that there are two block intersection sizes,
depending on whether the blocks are adjacent or not. So $\D_{\G}$ is a quasi-symmetric
design and $\G_2$ is one of its (strongly regular) block graphs. We obtain the below
proposition. Here we use the notation that is common for quasi-symmetric designs (cf. \cite{SS}, \cite{Handbook}).
That is, $\D_{\G}$ is a $2$-$(v,k,\lambda)$ design with replication number $r=\lambda \frac{v-1}{k-1}$ and two block intersection sizes $x$ and $y$. We do however not make the usual convention that $y>x$. The corresponding
block graph $\G_2$, where two blocks are adjacent if they intersect in $y$ points, is strongly regular with parameters $(b,a,c,d)$, with $b=\frac{vr}k$, $a=\frac{(r-1)k-x(b-1)}{y-x}$, $d=a+\rho_1\rho_2$, $c=d+\rho_1+\rho_2$. Here $\rho_1=\frac{r-\lambda-k+x}{y-x}$ and $\rho_2=\frac{x-k}{y-x}$ are the (usual) non-trivial eigenvalues of $\G_2$.
An important property in the following is that for a point-block pair $(P,B)$, the number of
blocks $B' \neq B$ incident with $P$ and intersecting $B$ in $y$ points equals $\frac{(\lambda-1)(k-1) -(x-1)(r-1)}{y-x}$ or $\frac{\lambda k -xr}{y-x}$, depending on whether $P$ is incident to $B$ or not, respectively (cf. \cite[Thm. 3.2]{GS}).

\begin{prop}\label{quasi1} Let $\G$ be a biregular graph with valency partition
$(V_1,V_2)$ such that $\G_1$ is empty, the edges between
$V_1$ and $V_2$ form the incidence relation of a quasi-symmetric design
$\D_{\G}$, and $\G_2$ is the corresponding block graph, with notation as above.
Then $\G$ has three $\mathcal{L}$-eigenvalues $0,\theta_1,\theta_2$ if and
only if
$$\begin{array}{rcl}
\frac{r}{k+a}&=&tr^2-(\theta_1-1)(\theta_2-1)r,\\
\frac{\lambda}{k+a}&=&tr^2,\\
\frac kr+\frac{a}{k+a}&=&t(k+a)^2-(\theta_1-1)(\theta_2-1)(k+a),\\
\frac{x}{r}+\frac{d}{k+a}&=&t(k+a)^2,\\
\frac{y}{r}+\frac{c}{k+a}&=&t(k+a)^2+2-\theta_1-\theta_2,\\
\frac{\lambda k-xr}{(y-x)(k+a)}&=&tr(k+a),\\
\frac{(\lambda-1)(k-1)-(x-1)(r-1)}{(y-x)(k+a)}&=&tr(k+a)+2-\theta_1-\theta_2,
\end{array}$$
where $t=\frac{\theta_1\theta_2}{vr+b(k+a)}$.
\end{prop}
\begin{proof}{ This follows immediately from Theorem \ref{main}.
}
\end{proof}
Any Steiner system, i.e., a $2$-$(v,k,1)$ design,
is a quasi-symmetric design (if $b > v$) with $y=1$, $x=0$, $r=\frac{v-1}{k-1}$, and block graph
with parameters $(\frac{v(v-1)}{k(k-1)},(r-1)k,r-2+(k-1)^2,k^2)$. These parameters satisfy the
above conditions and so each Steiner system gives a graph with three $\L$-eigenvalues, the non-trivial
ones being $1+\frac{1}{r}$ and $1-\frac1k+\frac1{rk}$.
The $2$-$(v,2,1)$
design of all pairs gives a graph that can also be obtained from
the triangular graph $T(v+1)$ by removing all edges in a maximal
clique. Note also that the graphs of Proposition \ref{projective} are degenerate cases of this construction.

Another large family of quasi-symmetric designs, the multiples of symmetric designs (i.e., each block is repeated the same number of times) do not satisfy the conditions.

Among the residuals of biplanes, only the (three) $2$-$(10,4,2)$ designs satisfy the above conditions, with $x=2$ and $y=1$, and give graphs on $25$ vertices with three $\L$-eigenvalues $0,\frac56, \frac43$. The graph $\G_2$ is the triangular graph $T(6)$.

Another example is obtained from the unique quasi-symmetric $2$-$(21,6,4)$ design with $b=56, r=16, x=2, y=0$.
Here $\G_2$ is the Gewirtz graph, and $\G$ has $\L$-eigenvalues $0,\frac78, \frac{11}8$. Unfortunately or not, this graph on 77 vertices is strongly regular, as is well-known, cf. \cite{GS, BH}.

Instead of taking $\G_1$ empty, we now let it be
complete. Also in this case the graph between $V_1$ and $V_2$ is the
incidence graph of a $2$-design $\D_{\G}$,
and we find some new examples by considering the case that $\D_{\G}$ is a quasi-symmetric
design and $\G_2$ is a strongly regular graph corresponding
to $\D_{\G}$. The following is the analogue of Proposition \ref{quasi1}.

\begin{prop} Let $\G$ be a biregular graph with valency partition
$(V_1,V_2)$ such that $\G_1$ is complete, the edges between
$V_1$ and $V_2$ form the incidence relation of a quasi-symmetric design
$\D_{\G}$, and $\G_2$ is the corresponding block graph, with notation as above.
Then $\G$ has three $\mathcal{L}$-eigenvalues $0,\theta_1,\theta_2$ if and
only if
$$\begin{array}{rcl}
\frac{v-1}{v-1+r}+\frac{r}{k+a}&=&t(v-1+r)^2-(\theta_1-1)(\theta_2-1)(v-1+r),\\
\frac{v-2}{v-1+r}+\frac{\lambda}{k+a}&=&t(v-1+r)^2+2-\theta_1-\theta_2,\\
\frac k{v-1+r}+\frac{a}{k+a}&=&t(k+a)^2-(\theta_1-1)(\theta_2-1)(k+a),\\
\frac{x}{v-1+r}+\frac{d}{k+a}&=&t(k+a)^2,\\
\frac{y}{v-1+r}+\frac{c}{k+a}&=&t(k+a)^2+2-\theta_1-\theta_2,\\
\frac{k}{v-1+r}+\frac{\lambda k-xr}{(y-x)(k+a)}&=&t(v-1+r)(k+a),\\
\frac{k-1}{v-1+r}+\frac{(\lambda-1)(k-1)-(x-1)(r-1)}{(y-x)(k+a)}&=&t(v-1+r)(k+a)+2-\theta_1-\theta_2,
\end{array}$$
where $t=\frac{\theta_1\theta_2}{v(v-1+r)+b(k+a)}$.
\end{prop}
A multiple of a projective plane, i.e., a design obtained from a projective plane by repeating
each line $\lambda$ times, is a quasi-symmetric design with parameters
2-$(k^2-k+1,k,\lambda)$, with $r=\lambda k$, $x=1$, $y=k$, $a=\lambda-1$,
and it satisfies the above conditions. Here $\G_2$ is a disjoint union of cliques of size $\lambda$, and the obtained graph $\G$ has non-trivial $\L$-eigenvalues $\frac{v}{k^2+k(m-1)}$ and $1+\frac1{k+m-1}$. This construction is again a generalization of the construction in
Proposition \ref{projective}.

Other attempts to construct biregular graphs with three $\L$-eigenvalues could be inspired by the papers by Haemers and Higman on strongly regular graphs with a strongly regular decomposition \cite{HH} and by Higman on strongly regular designs \cite{H}, but we have not worked this out.

\section{Cones}

A {\sl cone} over a graph $\G'$ is a graph obtained by adjoining a new
vertex to all vertices of $\G'$, i.e., it is a graph which has a vertex of valency $n-1$.

\begin{lem}\label{conebi} Let $\G$ be a cone over $\G'$.
If $\G$ has three $\mathcal{L}$-eigenvalues then $\G'$ is regular or biregular,
and the valency partition of $\G$ is equitable.
\end{lem}

\begin{proof} {Let $v$ be a vertex of valency $n-1$, and $W$ be the set of remaining vertices, so that
$\G'$ is the induced graph on $W$. From Theorem \ref{main} we
find that $\hat{d}_w=td_w^2-(\theta_1-1)(\theta_2-1)d_w$ and
$\hat{\lambda}_{wv}=td_w(n-1)+2-\theta_1-\theta_2$ for $w \in
W$. Because in this case
$\hat{d}_w=\hat{\lambda}_{wv}+\frac1{n-1}$, we obtain a
quadratic equation for $d_w$, which shows that $\G'$ is regular
or biregular. That the valency partition is equitable can be
proven in a similar way as in Lemma \ref{biregular}.}
\end{proof}

\begin{prop}\label{cone} Let $\G$ be a cone over a regular graph $\G'$.
Then $\G$ has three $\mathcal{L}$-eigenvalues if and only if
$\G'$ is a disjoint union of (at least two) cliques of the same
size $d$, say. In this case, the non-trivial $\L$-eigenvalues
are $\frac1d$ and $1+\frac1d$.
\end{prop}

\begin{proof} {As before, let $v$ be a vertex of valency $n-1$, and $W$ be the set of remaining vertices,
which now have constant valency $d$, say. If $\G$ has three $\mathcal{L}$-eigenvalues, then by
Theorem \ref{main}, $\G'$ is a strongly regular graph with parameters $(n-1,d-1,\lambda,\mu)$, where
$\frac1{n-1}+\frac{\lambda}d=td^2+2-\theta_1-\theta_2$ (the
normalized number of common neighbors of two adjacent vertices
$w,w'\neq v$) and $\frac1{n-1}+\frac{\mu}d=td^2$ (the
normalized number of common neighbors of two non-adjacent vertices
$w,w'\neq v$).

Moreover, by combining
$\frac{n-1}{d}=t(n-1)^2-(\theta_1-1)(\theta_2-1)(n-1)$ (the
normalized valency of $v$) and
$\frac{d-1}{d}+\frac1{n-1}=td^2-(\theta_1-1)(\theta_2-1)d$ (the
normalized valency of $w \in W$), we obtain that
$t=\frac1{(n-1)d^2}$, and this implies that $\mu=0$. Thus,
$\G'$ is a disjoint union of cliques of size $d$. Therefore
$\lambda=d-2$, and the above equations now show that
$\{\theta_1,\theta_2\}=\{\frac1d,1+\frac1d\}$.

On the other hand, by checking all equations in Theorem
\ref{main}, it follows that the cone over a disjoint union of
$d$-cliques indeed has three $\L$-eigenvalues.}
\end{proof}
Examples of cones over biregular graphs can be constructed
using certain strongly regular graphs, as we shall see next.
Recall that a {\sl conference graph} is a strongly regular
graph with parameters $(n,k,\lambda,\mu)$ with $n=2k+1$,
$k=2\mu$, and $\lambda=\mu-1$.

\begin{prop}\label{valency1} Let $\Gamma$ be a graph with minimum valency one.
Then $\Gamma$ has three $\L$-eigenvalues if and only if it is a star
graph or a cone over the disjoint union of an isolated vertex and
a conference graph. The latter has non-trivial $\L$-eigenvalues $\frac{n \pm \sqrt{n-2}}{n-1}$,
each with multiplicity $\frac{n-1}2$.
\end{prop}

\begin{proof} {Suppose that $\Gamma$ has three $\L$-eigenvalues, and $n$
vertices. Let $u$ be a vertex of valency $d_u=1$, and let $v$ be
its neighbor. Because the diameter of $\Gamma$ is two, it follows
that every other vertex is adjacent to $v$. So $\Gamma$ is a cone,
say over $\Gamma'$. If $\Gamma'$ is regular, then $\Gamma$ is a star graph
by Proposition \ref{cone}. So let's assume that $\Gamma'$ is not
regular, and hence is not empty.

Using that $\frac1{n-1}=\hat{\mu}_{uw}=td_w$ for all $w \neq u,v$,
we obtain that $d_w=\frac1{t(n-1)}=:d$ is the same for all $w \neq
u,v$. By combining
$\frac1{n-1}=\hat{d}_u=t-(\theta_1-1)(\theta_2-1)$ and
$\frac1{n-1}+\frac{d-1}d=\hat{d}_w=td^2-(\theta_1-1)(\theta_2-1)d$,
we find that $d=(n-1)/2$.

It is straightforward now to show that the induced graph on the
vertices except $u$ and $v$ is strongly regular with parameters
$(n-2,d-1,\lambda,\mu)$, where
$\frac1{n-1}+\frac{\lambda}d=td^2+2-\theta_1-\theta_2$ (the
normalized number of common neighbors of two adjacent vertices
$w,w'\neq u,v$), and $\frac1{n-1}+\frac{\mu}d=td^2=\frac12$ (the
normalized number of common neighbors of two non-adjacent vertices
$w,w'\neq u,v$). Using the above and the equation
$0=\hat{\lambda}_{uv}=t(n-1)+2-\theta_1-\theta_2$, this implies
that $\lambda=(d-3)/2$ and $\mu=(d-1)/2$.
Thus, we have found that $\Gamma'$ is the disjoint union of an
isolated vertex and a conference graph.

On the other hand, the star graph is complete bipartite, so has
three $\L$-eigenvalues. Also the cone over the disjoint union of an
isolated vertex and a conference graph has three $\L$-eigenvalues; the
non-trivial ones being $\frac{n \pm \sqrt{n-2}}{n-1}$ (this
follows from the above equations and Theorem \ref{main}), each with multiplicity $\frac{n-1}2$.
 }\end{proof}
We thus have examples where the multiplicities of the
non-trivial $\mathcal{L}$-eigenvalues are the same. We finish
this paper at the other extreme, by identifying the graphs
where one non-trivial $\mathcal{L}$-eigenvalue is simple.

\begin{prop}\label{2simple}
Let $\G$ be a graph with three $\mathcal{L}$-eigenvalues, of which two
are simple. Then $\G$ is either complete bipartite or
a cone over the disjoint union of two cliques of the same size.
\end{prop}
\begin{proof} {Let $\theta$
be the $\L$-eigenvalue with multiplicity $n-2$. So the
rank of $\mathcal{L}-\theta I$ is two. First, assume that $\theta=1$.
Using Lemma \ref{basic}, it follows that the $\mathcal{L}$-spectrum
of $\G$ is $\{[0]^1, [1]^{n-2}, [2]^{1}\}$, and hence by Proposition
\ref{bipartite}, $\G$ is complete bipartite.

Next, assume that $\theta \neq 1$. By considering principal submatrices
of $\mathcal{L}-\theta I$ of size three, it follows that $\G$ has no cocliques of size three.
For a vertex $u$, let $R_u$ be
the corresponding row in $\L-\theta I$. Consider now two vertices $u$ and $w$ that
are not adjacent. Then $R_u$ and $R_w$ span the row space of $\L-\theta I$.
Let $v$ be a common neighbor of $u$ and $w$. Then
$$(1-\theta)R_v=-\frac1{\sqrt{d_ud_v}}R_u-\frac1{\sqrt{d_wd_v}}R_w.$$
This implies that if $z$ is any fourth vertex --- which is
adjacent to at least one of $u$ and $w$ --- is adjacent to $v$.
So $v$ is adjacent to all other vertices; $d_v=n-1$.

We claim now that $v$ is the only common neighbor of $u$ and
$w$. To show this claim, suppose that $v'$ is another common
neighbor; hence also $d_{v'}=n-1$. Then applying the above
equation to entries corresponding to $v$ and $v'$ shows that
$(1-\theta)^2 =
-(1-\theta)\frac1{n-1}=\frac1{d_u(n-1)}+\frac1{d_w(n-1)}$. This
implies that $\theta=\frac n{n-1}$, which implies that $\G$ is
a complete graph by (ii) and (iii) of Lemma \ref{basic}; a
contradiction.

Because of Proposition \ref{valency1}, both $u$ and $w$ are not vertices with valency one.
Therefore there are vertices that are adjacent to one, but not the other.
If $z$ is a vertex that is adjacent to $u$, but not to $w$, then
$$(1-\theta)R_z=-\frac1{\sqrt{d_ud_z}}R_u,$$ which implies that any vertex different from $u$ and $z$
is adjacent to $u$ if and only if it is adjacent to $z$, and so $d_u=d_z$. Moreover, it tells us
that $d_u=\frac1{\theta-1}$ (consider the entry corresponding to $v$ in the above equation). Of course,
the situation where $u$ and $w$ are interchanged is completely the same. It thus follows that
$\G$ is a cone over the disjoint union of two cliques of the same size.
}\end{proof}

\section{Concluding remarks}
By computer, we checked all connected graphs with at most 10
vertices, and millions of graphs with 11 or 12 vertices and
diameter two for having three normalized Laplacian eigenvalues.
In this way, we obtained only a few nonregular nonbipartite
examples; all of these can be constructed by the methods in
this paper.

However, a classification of all graphs with three normalized
Laplacian eigenvalues still seems out of reach. In this paper,
we gave a combinatorial characterization that turned out to be
useful in such a classification within some very special
classes of graphs. In future work, it seems interesting also to
consider graphs with more distinct valencies, or to find an
upper bound on the number of distinct valencies in graphs with
three normalized Laplacian eigenvalues.

Finally, we mention that after submitting this paper, we were
informed that some of our results were obtained also by Cavers
\cite{Cth}.

\noindent {\bf Acknowledgements.} We thank a referee for some
very useful suggestions, and Kris Coolsaet for writing a
version of nauty's graph generator geng that filters graphs
with diameter two.

{\small

\bibliographystyle{siam}
}
\end{document}